\documentclass[11pt,a4paper]{article}

\usepackage[utf8]{inputenc}
\usepackage[T1]{fontenc}
\usepackage{lmodern}          
\usepackage{microtype}        
\usepackage{geometry}         
\usepackage{xcolor}

\usepackage{graphicx}
\usepackage{amsmath}
\usepackage{amssymb}
\usepackage{amsfonts}
\usepackage{amsthm}
\usepackage{mathrsfs}
\usepackage{enumitem}
\setlist{nosep}

\usepackage{authblk}

\theoremstyle{plain}
\newtheorem{theorem}{Theorem}[section]
\newtheorem{lemma}[theorem]{Lemma}

\newtheorem{proposition}[theorem]{Proposition}

\theoremstyle{definition}
\newtheorem{definition}[theorem]{Definition}
\newtheorem{example}[theorem]{Example}
\newtheorem{settings}[theorem]{Setting}

\newtheorem{remark}[theorem]{Remark}

\theoremstyle{remark}

\usepackage{hyperref}
\usepackage[nameinlink,capitalize]{cleveref}

\newcommand{\Z}{\mathbb{Z}}
\newcommand{\Q}{\mathbb{Q}}
\newcommand{\R}{\mathbb{R}}

\newcommand{\spa}{\mathrm{span}}
\newcommand{\conv}{\mathrm{conv}}

\newcommand{\ve}{\varepsilon}

\newcommand{\co}{\mathcal O}
\newcommand{\mc}{\mathcal}
\newcommand{\keywords}[1]{\noindent\textbf{Keywords:} #1}
\newcommand{\subjclass}[2]{\noindent\textbf{#1 Mathematics Subject Classification:} #2.}

\title{The addition on totally positive integers uniquely determines the totally real number field}
\author{V\' \i t\v ezslav Kala\thanks{vitezslav.kala@matfyz.cuni.cz, ORCID: 0000-0001-5515-6801} }
\author{Ritoprovo Roy\thanks{{ritoprovo.roy@matfyz.cuni.cz, ORCID: 0000-0002-5958-3816}}}
\affil{Charles University, Faculty of Mathematics and Physics, Department of Algebra, Sokolov\-sk\' a 83, 18600 Praha~8, Czech Republic
}
\date{}

\begin{document}
\maketitle

\begin{abstract} Let $K$ be a totally real number field.
 We prove that the totally positive algebraic integers $\co_K^+$ of $K$, viewed as an abstract additive semigroup, uniquely determine $K$. We also describe all the additive relations by giving a presentation of $\co_K^+$ in terms of indecomposables as generators and a certain concrete set of generating relations. 
 In fact, we obtain this presentation in a more general class of discrete semigroups which we call semilattices.

 \medskip
 \keywords{totally real number field, algebraic integers, semigroup presentation}
 
 \medskip
 \subjclass{2020}{11H06, 11R04, 11R80, 20M05, 20M14}

 \medskip
 
 \noindent\textbf{Funding:} This publication was produced with the financial support of the European Union and the Ministry of Education, Youth and Sports, under the Jan Amos Komenský Operational Programme, project Returns at CU (Reg. No. CZ.02.01.01/00/24\_037/0013839), grant RTNS25-044 (supporting R.R.).

 \medskip
 
 \includegraphics[width=0.8\linewidth]{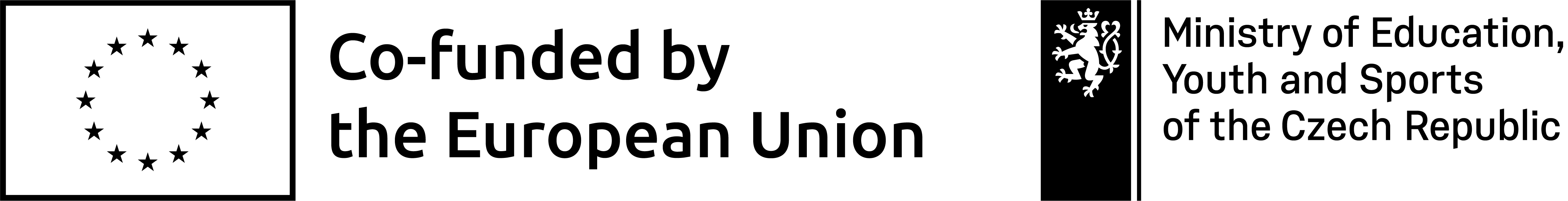}
 
 \medskip
 
 \noindent\textbf{Declarations:} The authors have no competing interests to declare that are relevant to the content of this article. Data sharing not applicable as no datasets were generated or analyzed during the current study.
 
 \medskip
 
 \noindent
 \textbf{Acknowledgments:}  
 The LLM Claude Opus 5 was used to critically review drafts of the manuscript; the suggested corrections were then carried out by the authors. All the text and contents of the article are human-generated, with the exception of Remark \ref{rem:counterex} which describes results discovered by Claude Opus 5.
 
\end{abstract}

\section{Introduction}\label{sec:intro}

Let $K$ be a totally real number field with ring of integers $\co_K$. Let $\co_K^+$ denote the semigroup of totally positive algebraic integers. In analogy with irreducible elements that cannot be factored multiplicatively, we say that a totally positive algebraic integer $\alpha\in\co_K^+$ is \textit{indecomposable} if it cannot be additively decomposed as the sum $\alpha=\beta+\gamma$ of two totally positive algebraic integers $\beta,\gamma\in\co_K^+$. As indecomposables generate  $\co_K^+$, they are the building blocks that are the first key to understanding its structure.

It is curious that, in contrast with the huge interest in irreducible elements and multiplicative factorizations, additively indecomposable elements have not received much attention. One reason might be that in the ring $\Z$ of integers in $K=\Q$, the structure of indecomposables is trivial, with 1 being the only indecomposable.

Dress--Scharlau \cite{DS} proved that in a real quadratic field $\Q(\sqrt D)$, the indecomposables correspond precisely to certain ``semiconvergents'' to the continued fraction of $\sqrt D$ (or $(1+\sqrt D)/2$), see \cite[$\S 16$]{Pe}. Based on their result, Narkiewicz \cite[Problem 53]{Na} formulated an open problem concerning the structure of indecomposables, namely, to determine the maximum norm of an indecomposable in a given field $K$. Brunotte \cite{Br} gave an upper bound in terms of the units of $K$, and recently Kala--Yatsyna \cite{KY3} proved that the discriminant itself works as an upper bound. Based on the evidence from quadratic \cite{DS, TV}, cubic \cite{KT, Ti}, and multiquadratic \cite{KM, Man} number fields, the discriminant bound seems to be close to best possible.

The recent renewed interest of Kala, Man, Tinkov\' a, Yatsyna, and others in indecomposables stems from their connection to universal quadratic forms. This was first applied by Siegel \cite{Si3}, and then much later rediscovered by Blomer--Kala \cite{BK}. Generalizing the sum of four squares $X^2+Y^2+Z^2+W^2$ that represents all positive integers $\Z^+$, a totally positive definite quadratic form over $\co_K$ is \textit{universal} if it represents all of $\co_K^+$. It turned out to be possible to control the properties of universal quadratic forms, such as their minimal ranks \cite{BK, KT, KYZ, Ya} and characterizations through criterion sets \cite{BH, CO,KKR}, in terms of the additive structure of $\co_K^+$ and its indecomposables. For more information, see the survey \cite{Ka}.

Despite all the aforementioned recent works, we still lack a good understanding of the indecomposables and the additive semigroup $\co_K^+$, especially when the degree $[K:\Q]$ is large.
Our starting point is the real quadratic case $K=\Q(\sqrt D)$. Besides the description of indecomposables due to Dress--Scharlau \cite{DS}, Hejda--Kala \cite{HK} described all the additive relations that hold between indecomposables, and used them to further show that $\co_{\Q(\sqrt D)}^+$, viewed as an abstract semigroup, uniquely determines the real quadratic field $\Q(\sqrt D)$.

This result is rather surprising! First of all, the additive group $\co_K$ determines only the degree $n=[K:\Q]$, as it is isomorphic to $\Z^n$ thanks to the existence of integral basis. Second, there has been a lot of interest  in determining which invariants uniquely determine the number field: For example, the absolute Galois group does \cite{Neu,Uch}, but the Dedekind zeta-function \cite{CdSL+,MS,Pr,Su} or the ring of adeles \cite{Ko} do not.

In their work, Hejda--Kala \cite{HK} heavily relied on the description of indecomposables in terms of the continued fraction \cite{DS}. They posed the open question about the higher degree generalization, but wrote that ``\textit{this may be quite hard, as we lack a good understanding of indecomposable elements}''.
Nevertheless, as our first main result, we answer their question in the affirmative.

\begin{theorem}\label{th:isom}
	Let $K, L$ be totally real number fields. Then
	\[K\simeq L \text{ as fields } \Longleftrightarrow \ \co_K^+\simeq \co_L^+ \text{ as additive semigroups.}
	\]
\end{theorem}

The main tool behind the proof of the theorem is our description of additive relations between totally positive elements in a general totally real number field.
In fact, our arguments concerning such relations work more generally than just in the additive semigroup $\co_K^+$. We thus introduce the notion of a \textit{semilattice} $\mc S$ as a certain discrete subsemigroup of $\R^n$ (see Definition \ref{de:semilattice}). A similar generalization was very recently considered by Fukshansky--Kostopoulou \cite{FK} who connected the properties of indecomposables to the geometry of the corresponding lattice, such as its successive minima, and extended some of the results of \cite{HK} to other semilattices in $\R^2$. Pěchoučková--Stern--Zindulka \cite{PSZ} recently considered partitions in another general class of semigroups.
Also note that when $n=1$, semilattices in $\R^1$ are precisely numerical semigroups whose study constitutes a highly active research area \cite{Cu, De, RB, RG, Wi}.

After covering some preliminaries in Section \ref{sec:2}, we further discuss that a general semilattice need not contain any indecomposables. Regardless, we can construct certain relations \eqref{eq:Ra} in the semilattice (in Section \ref{sec:3}). After summarizing some preliminaries on totally real number fields in Section \ref{sec:4}, we use our relations  in Section \ref{sec:5} to prove Theorem \ref{th:isom}. The main idea is that, instead of the explicit work involving continued fractions done by \cite{HK}, we reconstruct the number field through the convex geometry of the cone of totally positive elements.

The goal of Sections \ref{sec:6} and \ref{sec:7} is to generalize \cite[Corollary 3]{HK} by proving that our relations \eqref{eq:Ra} generate all the relations in the semilattice $\mc S$, provided that the indecomposables additively generate all of $\mc S$. In order to express this result formally, we adopt the language of universal algebra and specifically the notion of a presentation \cite[Chapter I, Section 6]{Gr}. This is introduced at the beginning of Section \ref{sec:6} where we then prove our second main result, Theorem \ref{th:pres1}, which says that indeed, the indecomposables together with the relations \eqref{eq:Ra} give a certain presentation of $\mc S$.

However, a typical semilattice (such as the totally positive integers $\co_K^+$ when $K\neq\Q$) contains infinitely many indecomposables, in which case Theorem \ref{th:pres1} also involves infinitely many relations, one for each indecomposable, which is quite unsatisfactory. Fortunately, in $\co_K^+$, there are only finitely many classes of indecomposables up to the multiplication by totally positive units (thanks to the aforementioned fact that all indecomposables have bounded norms). We conclude the paper by Theorem \ref{th:pres2} in which we construct a corresponding set of relations that is finite up to the multiplication by totally positive units and that still generates all the relations.
This result should be contrasted with the work \cite{AF+} that concerned a certain combinatorial variant of this problem, in which the authors showed that it is sometimes \textit{not} possible to generate all the relations. Our Theorem \ref{th:pres2} indeed gives a ``local'' set of generating relations, answering a question that was informally posed in \cite[p. 571]{AF+}.

\section{Semilattices} \label{sec:2}

A \textit{commutative semigroup} is a set $S$ equipped with a binary operation + that is commutative and associative. Moreover, such $S$ is 
\begin{itemize}
	\item \textit{cancellative} if for all $x,y,z\in S$ we have $x+z=y+z\implies x=y$,
	\item \textit{torsion-free} if for all $x,y\in S$, $k\in\Z^+$ we have $kx=ky\implies x=y$.
\end{itemize}

We will be interested in a special class of cancellative, torsion-free, commutative semigroups that we will call semilattices. However, in order to introduce them, we need to recall some terminology first.

\medskip

Let $V$ be an $\R$-vector space of dimension $n\in\Z^+$. We consider $V$ to be equipped with the usual topology on $V\simeq \R^n$.

We denote by
\begin{itemize}
	\item $\spa_\R(X)$ the \textit{$\R$-linear span} of $X$, i.e., the smallest $\R$-vector subspace of $V$ that contains~$X$,
	\item $\spa_\Z(X)$ the \textit{$\Z$-linear span} of $X$, i.e., the smallest $\Z$-submodule of $V$ that contains $X$,
	\item $\conv(X)$ the \textit{convex hull} of $X$, i.e., the smallest convex subset of $V$ that contains $X$,
	\item $\overline X$ the \textit{closure} of $X$, i.e., the smallest closed subset of $V$ that contains $X$.
\end{itemize}

\medskip

A discrete subgroup $\Lambda\subset V$ is a \textit{lattice} if $\spa_{\R}(\Lambda)=V$.
(Note that we are assuming that all our lattices have full rank, which one can of course achieve without loss of generality by replacing $V$ by $\spa_{\R}(\Lambda)$.)

\begin{definition} \label{de:semilattice}
	Let $V$ be an $\R$-vector space of dimension $n$. A subset $\mc S\subset V$ is a \textit{semilattice} if 
	$\mc S$ is a semigroup, $0\not\in \mc S$, and  $\spa_\Z(\mc S)$ is a lattice.
\end{definition}

Clearly, a semilattice $\mc S$ is a cancellative, torsion-free, commutative semigroup that is discrete as a subset of $V$. We also have $V=\spa_\R(\mc S)$.

\begin{example}
	Let $V=\R$. A subsemigroup $\mc S$ of $\Z^+$ is called a \textit{numerical semigroup} if the complement $\Z^+\setminus\mc S$ is finite (note that sometimes it is assumed that $0$ must lie in a numerical semigroup, which we are not allowing here; this is without loss of generality, for one can always just consider $\mc S\cup\{0\}$). Every numerical semigroup is a semilattice according to our definition. 
	
	Conversely, if $\mc S\subset \R$ is a semilattice, then there is a group isomorphism $\iota:\spa_\Z(\mc S)\simeq \Z$ (because $\Z$ is the only lattice in $\R$, up to isomorphism), and we can choose this isomorphism so that $\iota(\mc S)\subset \Z^+$ (since $0$ does not lie in a semilattice, $\iota(\mc S)$ can not contain both positive and negative elements). Then $\iota(\mc S)\subset \Z^+$ is a numerical semigroup.
	
	In other words, semilattices in $\R^1$ are exactly numerical semigroups, up to semigroup isomorphism.	
\end{example}

\begin{remark}
	Let us remark that in semigroup theory, a \textit{semilattice} usually means something different from our Definition \ref{de:semilattice}, namely, an idempotent (i.e., $x+x=x$ for all $x$) commutative semigroup. This definition is motivated by the notion of \textit{lattice} in universal algebra and order theory (i.e., a partially ordered set equipped with the meet $\wedge$ and join $\vee$ operations). 
	
	Thus, \textit{lattice} unfortunately means two very different things (in the number-theoretic and universal-algebraic contexts), and we similarly introduce the term \textit{semilattice} in the number-theoretic setting.	
\end{remark}

In order to define another important class of semilattices, let us introduce some further terminology.
We say that a subset $X\subset V$ is
\begin{itemize}
	\item \textit{convex} if $\forall x,y\in X, \forall s,t\in \R_{\geq 0}$ such that $s+t=1$ we have $sx + ty \in X$,
	\item \textit{closed under addition} if $\forall x,y\in X$ we have $x + y \in X$,
	\item \textit{closed under rays} if $\forall x\in X, \forall t>0$ we have $tx \in X$.
	\item \textit{closed under outward rays} if $\forall x\in X, \forall t\geq 1$ we have $tx \in X$.
\end{itemize}

A subset $\mc{C}\subset V$ is a \textit{cone} if it is convex, closed under addition and outward rays, and $0\not\in \mc C$.

\medskip
    
Note that our definition is quite general, as we are not requiring a cone to be closed and we are not requiring it to be closed under rays. Also, the condition that $0\not\in \mc C$ is not always assumed when defining a cone, but for us it will be important in order to elegantly exclude some trivial cases.

\begin{example} \label{set:1}
     Let $V$ be a $\mathbb{R}$-vector space of dimension $n$, $\Lambda$ a lattice in $V$, and $\mathcal{C}$ a cone in $V$. Then $\mathcal{S} := \mathcal{C} \cap \Lambda$ is a semilattice in $\spa_\R(\mc S)$ (provided that $\mc S\neq\emptyset$). 
 
	Note that in this case, we can assume that $V=\spa_\R(\mc S)$ without loss of generality, for if it were not satisfied, we could replace $V$ by $\spa_\R(\mc S)$, $\Lambda$ by $\Lambda\cap\spa_\R(\mc S)$, and $\mc C$ by $\mc C\cap\spa_\R(\mc S)$.
\end{example}

\begin{definition} Let $\mc S\subset V$ be a semilattice.
    An element $\alpha \in \mathcal{S}$ is  \textit{indecomposable} if there do not exist $\beta,\gamma \in \mathcal{S}$ such that $\alpha = \beta + \gamma$. Let us denote by $\mc I$ the set of all indecomposables in $\mc S$.
\end{definition}

\begin{proposition} \label{prop:sum ind}  Let $\mc S\subset V$ be a semilattice.
	If $0 \notin \overline{\conv(\mc{S})}$, then every element of $\mathcal{S}$ is a finite sum of indecomposables.
\end{proposition}
\begin{proof}
	We will use the hyperplane separation theorem (see e.g. \cite[Corollary 11.4.2]{Ro}) for the 
	sets $\{0\}$ and $ \overline{\conv(\mathcal{S})}$. They satisfy its assumptions as both sets are non-empty, convex, and closed, they are disjoint, and $\{0\}$ is compact.
	The theorem then says that there exists a linear function $f: V \to \R$ and a positive real number $\delta$ such that $f(x) \geq \delta$ for all $x \in  \overline{\conv(\mathcal{S})}$, and $f(0)=0$.

	If now $\alpha \in \mathcal{S}$ decomposes as $\alpha = \beta_1 +  \dots + \beta_k$ for $\beta_i\in \mc S$, then 
	by linearity of $f$, we have that 
	$f(\alpha) = \sum_{i = 1}^{k}f(\beta_i) \geq k\delta$, and so the length $k$ of the decomposition is bounded (depending on $\alpha$). Thus we can consider the maximum $\ell(\alpha)$ of lengths of all possible decompositions of $\alpha$. Each element has the decomposition $\alpha=\alpha$, and so $\ell(\alpha)\in\Z^+$.

    We proceed by induction on $\ell(\alpha)$.
If $\ell(\alpha)=1$, then $\alpha$ is indecomposable.

Assume now that $\ell(\alpha)>1$ and that we know that every element of length $<\ell(\alpha)$ is a sum of indecomposables. 
As $\ell(\alpha)>1$, we have that $\alpha=\beta+\gamma$ is decomposable. Then $\ell(\beta),\ell(\gamma)<\ell(\alpha)$, and so $\beta,\gamma$ are sums of indecomposables by the induction hypothesis. Thus also $\alpha$ is a sum of indecomposables.
\end{proof}

\begin{example}
	In a general semilattice, not all elements need to be sums of indecomposables.
    
    As one example, we can take the lattice $\Lambda = \mathbb{Z}^2 \subset \mathbb{R}^2$ and define the cone $\mathcal{C} = \{(x,y) \in \mathbb{R}^2 \mid x > \sqrt{2}y \}$. Then $\mathcal{S} = \mathcal{C} \cap \mathbb{Z}^2 = \{(x,y) \in \mathbb{Z}^2 \mid x > \sqrt{2}y \}$ contains no indecomposables.

    As another example take the lattice $\Lambda = \mathbb{Z}^2 \subset \mathbb{R}^2$ and the cone $\mathcal{C} = \{(x,y) \mid (x = 0 \text{ and }  y > 0) \text{ or } (x>0)\}$. Then the only indecomposable in the semilattice $\mathcal{S} = \mathcal{C} \cap \mathbb{Z}^2$ is $(0,1)$; each of the elements $(0,y)$ for $y\in\Z^+$ is a sum of indecomposables, but no element $(x,y)$ with $x\in\Z^+$ is.
    
    However, our principal example of the semilattice $\co_K^+$ does satisfy the assumption of Proposition \ref{prop:sum ind}, see Section \ref{sec:4}.
   \end{example}
    
\begin{remark}\label{rem:counterex}
	While the criterion of Proposition \ref{prop:sum ind} is convenient, the converse implication does not hold.
	As an example (found by Claude Opus 5), one can take the semilattice $$\mc S:=\{(x,0)\in\Z^2\mid x\geq 1\}\cup\{(x,y)\in\Z^2\mid y\geq 1,x\geq -y^2\}\subset \R^2.$$
	Then it is not hard to check that $\mc I=\{(1,0)\}\cup\{(-y^2,y)\mid y\geq 1\}$ generates $\mc S$, but $0\in\overline{\conv(\mc S)}=\{(x,y)\in\Z^2\mid y\geq 0\}$.
	
	However, this semilattice is not of the form $\mc S=\mc C\cap \Lambda$ as in Example \ref{set:1}. 	
	For such semilattices $\mc S=\mc C\cap \Lambda$,
	Claude Opus 5 claims to have found a (rather non-trivial) proof of the converse implication {when the dimension is $n\leq 3$}, and counterexamples in all dimensions $\geq 4$. The write-up is available at \url{https://karlin.mff.cuni.cz/~kala/files/KR-appendix.pdf}.
\end{remark}

\section{Relations in a semilattice} \label{sec:3}

Let $\mc S\subset V$ be a semilattice. Since $V=\spa_\R(\mc S)$, we can fix an $\R$-basis of $V$ consisting of elements of $\mc S$, namely, let $\nu_1,\dots,\nu_n\in\mc S$ be an $\R$-basis of $V$.
Let us establish a key relation that expresses a multiple of each element in terms of this basis. As we are interested in semigroup relations, we will want to have all coefficients to be nonnegative integers.

\begin{lemma} \label{lem:rel1} 
	Let $\mc S\subset V$ be a semilattice and let $\nu_1,\dots,\nu_n\in\mc S$ be an $\R$-basis of $V$. 
	
	For each $\alpha \in \mathcal{S}$, there exist $m \in \Z^+$, $p_i,q_i\in\Z_{\geq 0}$  such that 
	\begin{equation*}
		m\alpha + \sum_{i = 1}^{n} q_i\nu_i = \sum_{i = 1}^{n} p_i\nu_i
		\tag{$\mathcal{\rho}_{\alpha}$},\label{eq:Ra}
	\end{equation*}
	\begin{itemize}
		\item for each $i=1,\dots,n$, $p_i$ and $q_i$ are not both nonzero (i.e., $p_iq_i=0$), and 
		\item $m,p_1,\dots,p_n,q_1,\dots,q_n$ are coprime (i.e., if $k\in\Z^+$ divides all of them, then $k=1$).
	\end{itemize}
 \end{lemma}

\begin{proof}
Let $\Lambda':=\spa_\Z(\nu_1,\nu_2,\dots,\nu_n)\subset \Lambda:=\spa_\Z(\mc S)$. As both lattices $\Lambda'\subset \Lambda$ have rank $n$, the index $k=(\Lambda:\Lambda')$ is finite.
  
For $\alpha \in \mathcal{S}\subset \Lambda$  we have $k\alpha\in\Lambda'$, and so there are $r_i\in\Z$ such that $k\alpha = \sum_{i = 1}^{n} r_i\nu_i$. Let us divide $k,r_1,\dots, r_n$ by their greatest common divisor to obtain 
$m\alpha = \sum_{i = 1}^{n} s_i\nu_i$ where $m\in\Z^+,s_1,\dots,s_n\in\Z$ are coprime.

Now if $s_i\geq 0$, then set $p_i=s_i$ and $q_i=0$; if $s_i<0$, then set $p_i=0$ and $q_i=-s_i$. Then all the conditions from Lemma \ref{lem:rel1} are satisfied.
  \end{proof}

The coefficients $m \in \Z^+$, $p_i,q_i\in\Z_{\geq 0}$ of course depend on $\alpha$. However, note that in the proof we have also obtained that $m$ is uniformly bounded, i.e., $m$ divides the index $(\Lambda:\Lambda')$  that is independent of $\alpha$.

\begin{lemma}\label{lem:rel2}
	Let $\mc S\subset V$ be a semilattice and let $\nu_1,\dots,\nu_n\in\mc S$ be an $\R$-basis of $V$. Let $\alpha \in \mathcal{S}$ and let $m,p_i,q_i$ be some integers as in Lemma $\ref{lem:rel1}$.
	
	Assume that some $m' \in \Z^+$, $p_i',q_i'\in\Z_{\geq 0}$  satisfy  
	\begin{equation*}
		m'\alpha + \sum_{i = 1}^{n} q_i'\nu_i = \sum_{i = 1}^{n} p_i'\nu_i.
	\end{equation*}
	Then there exist $\ell\in\Z^+,r_i \in\Z_{\geq 0} $ such that $m' = \ell m$ and $p_i' = \ell p_i + r_i$ and $q_i' = \ell q_i + r_i$.
	
	In particular, the numbers $m,p_i,q_i$ satisfying  the conditions of Lemma $\ref{lem:rel1}$ are unique.
\end{lemma}

\begin{proof}
    We have 
    \[
    \alpha=\sum_{i = 1}^{n} \frac{p_i-q_i}m\nu_i=\sum_{i = 1}^{n} \frac{p_i'-q_i'}{m'}\nu_i.
    \]
    As the representation in the basis $\nu_1,\dots,\nu_n$ of $V$ is unique, this implies     
    \[\frac{p_i-q_i}m=\frac{p_i'-q_i'}{m'} \text{ for all } i.
    \]
    By Lemma \ref{lem:rel1}, the integers $m,p_1-q_1,\dots,p_n-q_n$ are coprime. Thus $m$ is the least common multiple of the denominators of the fractions $\frac{p_i-q_i}m=\frac{p_i'-q_i'}{m'}$, and so  $m\mid m'$; let $\ell =\frac {m'}m\in\Z^+$.

    Substituting this in the equation above we get that 
    \[\frac{p_i-q_i}m=\frac{p_i'-q_i'}{\ell m},
    \]
    which implies $p_i'-\ell p_i=q_i'-\ell q_i$. Let $r_i=p_i'-\ell p_i=q_i'-\ell q_i\in\Z$, so that we have $p_i' = \ell p_i + r_i$ and $q_i' = \ell q_i + r_i$.
    
    To show that $r_i$ is nonnegative, we recall that $p_i,q_i$ cannot be both nonzero. If $p_i = 0$, then $r_i=p_i'\geq 0$; and similarly if $q_i = 0$, then $r_i = q_i' \geq 0$.
    
    ``In particular'' part: If moreover $p_i',q_i'$ are not both nonzero, then $r_i=0$ for all $i$. By the coprimality assumption we then get $\ell =1$.
\end{proof}

Thus, in the following we will often work with the \textit{unique} equation \eqref{eq:Ra} that we have defined for every $\alpha\in\mc S$.

\section{Number field semilattices}\label{sec:4}

Let $K$ be a totally real number field of degree $n=[K:\Q]$, i.e., there are $n$ real embeddings $\sigma_i:K\hookrightarrow\R$. The \textit{norm} and \textit{trace} of $\alpha\in K$ are $N(\alpha)=\sigma_1(\alpha)\cdots\sigma_n(\alpha)$ and $\text{Tr}(\alpha)=\sigma_1(\alpha)+\dots+\sigma_n(\alpha)$. 
An element $\alpha \in K$ is \textit{totally positive} if $\sigma_i(\alpha) > 0$ for all $1 \leq i \leq n$.

Let $\co_K$ be the ring of integers of $K$ and $\co_K^+$ the semigroup of totally positive integers. 
An element $\alpha\in\co_K^+$ is \textit{indecomposable} if $\alpha\neq\beta+\gamma$ for all $\beta,\gamma\in\co_K^+$. We denote the set of all indecomposables in $\co_K^+$ by $\mathcal J$.

The group of units of $K$ is $\co_K^{\times}$ and the subgroup of totally positive units is denoted $\co_K^{\times, +}$. By Dirichlet's unit theorem,  $\co_K^{\times}$ and $\co_K^{\times, +}$ are finitely generated abelian groups of rank $n-1$, with torsion subgroups $\{\pm 1\}$ and $\{1\}$, respectively. In particular, there is a fundamental system of totally positive units $\ve_1,\dots,\ve_{n-1}\in\co_K^{\times, +}$ that are free generators of $\co_K^{\times, +}$.

The group of totally positive units  $\co_K^{\times, +}$ acts on $\co_K^+$ and $\mc J$ by multiplication. As all indecomposables have bounded norms (see \cite{Br}, \cite[Theorem 5]{KY3}), there are only finitely many indecomposables up to multiplication by $\co_K^{\times, +}$.
These indecomposables have been determined in a number of cases, e.g., \cite{DS, KM, KT}. 

Consider now the Minkowski embedding $\sigma: K \hookrightarrow \mathbb{R}^n, \sigma(\alpha)=(\sigma_1(\alpha),\dots,\sigma_n(\alpha))$. 
We have that $\sigma(\co_K)\subset \R^n$ is a lattice and  $\mc S=\sigma(\co_K^+)=\sigma(\co_K)\cap (\R^+)^n$ is a semilattice, obtained as the intersection of the lattice $\sigma(\co_K)$ with the cone $(\R^+)^n$. Observe that we indeed have that $\spa_\Z(\mc S)=\sigma(\co_K)$ is a lattice.

We denote the set of indecomposables in $\mc S$ by $\mathcal I=\sigma(\mathcal J)$. 
It is well-known that every element in $\co_K^+$ and $\mc S$ is a finite sum of indecomposables, but one can also verify this using our Proposition \ref{prop:sum ind}, because we have $\overline{\conv(\mc S)}\subset \{(x_1,\dots,x_n)\in \R_{\geq 0}^n\mid x_1+\dots+x_n\geq 1\}$ (as each $\alpha\in\co_K^+$ has $\text{Tr}(\alpha)\geq 1$). This will be an important assumption in Settings \ref{set:2} and \ref{set:3} below, which is thus satisfied for $\mc S=\sigma(\co_K^+)$.

Again, the group of totally positive units  $\co_K^{\times, +}$ (or its image $\sigma(\co_K^{\times, +})$) acts on
$\mc S$ and $\mc I$ by multiplication. As there are finitely many indecomposables in $\co_K^+$  up to multiplication by $\co_K^{\times, +}$, there is a finite set $\mc I_0\subset \mc I$ that is a fundamental set for the action of $\co_K^{\times, +}$ on $\mc I$, i.e., for each $\alpha\in\mc I$ there are $\alpha_0\in\mc I_0$ and $\ve\in\co_K^{\times, +}$ such that $\alpha=\ve\alpha_0$. 
Note that here and later we use the term \textit{fundamental set} rather than \textit{fundamental domain} to emphasize that we are not requiring $\ve$ and $\alpha_0$ to be unique.

\section{Reconstructing the number field}\label{sec:5}

Let us now establish that the abstract semigroup $(\co_K^+,+)$ uniquely determines the number field $K$, i.e., Theorem \ref{th:isom}.

\begin{proof}[Proof of Theorem $\ref{th:isom}$] The implication ``$\Longrightarrow$'' is trivial, so let us prove the converse.
	As $\co_K^+$ is isomorphic to its image under the Minkowski embedding $\sigma$, we will work with the semilattice $\mc S=\sigma(\co_K^+)$. We need to show that $\mc S$, viewed as an abstract semigroup, determines $K$ uniquely up to isomorphism. We will do this in several steps.
	
	To avoid confusion, let us comment on our task explicitly: We know that $\mc S$ is defined from some totally real number field $K$, that it is a semilattice in some  $\R^n$, etc. But at the beginning, we do \textit{not} know the embedding $\mc S\subset \R^n$ (or even the degree $n$), we are given $\mc S$ only as an abstract semigroup, i.e., a set equipped with a binary operation $+$. Only using this information, we need to reconstruct the number field $K$.
	
	\medskip
	
	\textbf{(1) Determining the degree $n$ and choosing a basis.} 
\nopagebreak
	
	For a positive integer $t$, let us say that a $t$-tuple of elements $\nu_1,\dots,\nu_t\in \mc S$ is \textit{good} if it has the following property: For each $\alpha\in \mc S$, there are $m \in \Z^+$, $p_i,q_i\in\Z_{\geq 0}$ such that	
	$m\alpha + \sum_{i = 1}^{t} q_i\nu_i = \sum_{i = 1}^{t} p_i\nu_i$.
	
	From Lemma \ref{lem:rel1} we know that some good tuple exists (namely, a basis of $\R^n$, which has $n$ elements). As $\spa_\R(\mc S)=\R^n$, every good $t$-tuple generates the vector space $\R^n$, and so $t\geq n$.
	
	Thus, by taking a good tuple with smallest size, this size must equal the degree $n=[K:\Q]$. This shows that the degree $n$ is uniquely determined by the abstract semigroup $\mc S$. For the rest of the proof, fix some good $n$-tuple $\nu_1,\dots,\nu_n$; this forms an $\R$-basis $\nu_1,\dots,\nu_n\in \mc S$ of the vector space $\R^n$.
	
	By Lemmas \ref{lem:rel1} and \ref{lem:rel2}, we then know that for each $\alpha\in \mc S$, there are (essentially unique) $m \in \Z^+$, $p_i,q_i\in\Z_{\geq 0}$  such that 
	\begin{equation*}
		m\alpha + \sum_{i = 1}^{n} q_i\nu_i = \sum_{i = 1}^{n} p_i\nu_i
		\tag{$\mathcal{\rho}_{\alpha}$},\label{eq:Ra2}
	\end{equation*}

	\medskip
	
	\textbf{(2) Embedding into $\R^n$.}	
	\nopagebreak
	
	Given the basis $\nu_1,\dots,\nu_n$ that we fixed in step (1), let us define a semigroup homomorphism $\tau:\mc S\hookrightarrow\R^n$ that maps $\nu_i$ to the $i$th canonical basis vector $e_i=(0,\dots,0,1,0,\dots,0)$. More precisely, for each $\alpha\in\mc S$, we consider the unique equation \eqref{eq:Ra2} and define $\tau(\alpha)$ as the unique element of $\R^n$ that satisfies
	\begin{equation*}
		m\tau(\alpha) + \sum_{i}q_i e_i = \sum_{i} p_i e_i.
	\end{equation*}
	As $m\neq 0$, such $\tau(\alpha)$ indeed exists and is unique. 
	
	It is then straightforward to check that $\tau:\mc S\hookrightarrow\R^n$ is indeed an injective semigroup homomorphism (as it is the restriction of the $\R$-vector space isomorphism $\nu_i\mapsto e_i$).
	
	We now want to view $\tau(\mc S)$ as a semilattice in $\R^n$. For that purpose,
	let us define $\mc C=\R^+\conv(\tau(\mc S))=\{tx\mid t>0, x\in \conv(\tau(\mc S))\}$ and $\Lambda = \spa_\Z(\tau(\mc S))$. We have $\mc S=\sigma(\co_K^+)=\sigma(\co_K)\cap (\R^+)^n$ and  $\R^+\conv(\mc S)=(\R^+)^n$.	
	Thus we see that under the isomorphism $\tau^{-1}$, $\mc C$ maps to the cone $(\R^+)^n$ and $\Lambda$ maps to the lattice $\sigma(\co_K)$. Hence $\tau(\mc S)=\mc C\cap \Lambda$.
	Note that the cone $\mc C$ is closed under rays (not only under outward rays).

	\medskip
	
	\textbf{(3) Boundary rays.}
\nopagebreak
	
	Consider the closure $\overline{\mc C}$ of the cone $\mc C$ defined in step (2). By its \textit{extreme ray} we mean a subset $\mc F\subset \overline{\mc C}$ such that
	\begin{enumerate}[label=\alph*), nosep]
		\item $\text{dim(span}_\R(\mc F))=1$,
		\item $\mc F$ is convex and closed under rays,
		\item for all $a,b\in \overline{\mc C}$, we have $a+b\in\mc F\implies a,b\in\mc F$.
	\end{enumerate}
	Note that $0\in\overline{\mc C}$, and so thanks to condition c) we have $0\in\mc F$ for every extreme ray $\mc F$.
	
	The extreme rays of $\R_{\geq 0}^n=\overline{\R^+\conv(\mc S)}$ (i.e., those corresponding to the semilattice $\mc S=\sigma(\co_K^+)$) are precisely the coordinate half-axes 
	$$E_i:=\{(x_1,\dots,x_n)\in\R^n \mid x_1=\dots=x_{i-1}=0,x_i\geq 0, x_{i+1}=\dots=x_n=0\};$$ there are $n$ of them. 
	As $\mc S\simeq \tau(\mc S)$ (and this isomorphism is given by a restriction of an $\R$-vector space isomorphism), $\overline{\mc C}$ also has precisely $n$ extreme rays, let us denote them $v_1\R_{\geq 0},\dots,v_n\R_{\geq 0}$ for  vectors $v_1,\dots,v_n\in\R^n$ that are linearly independent.

	\medskip
	
	\textbf{(4) Solving equations.}\nopagebreak
	
	Let us now go back to the (so far not explicitly known) number field $K$. Recall that we fixed an $\R$-basis $\nu_1,\dots,\nu_n\in \mc S=\sigma(\co_K^+)$ of $\R^n$ and let us define $\omega_i:=\sigma^{-1}(\nu_i)\in \co_K^+$.
	
	We have $K=\omega_1\Q+\dots+\omega_n\Q$ and we can formally consider $\omega_1\R+\dots+\omega_n\R$ as an $\R$-vector space of dimension $n$. We have that
	$\sigma:K\hookrightarrow\R^n$ is defined as $\sigma(\alpha)=(\sigma_1(\alpha),\dots,\sigma_n(\alpha))$, where $\sigma_i:K\hookrightarrow\R$ are all the real embeddings. If we denote $\omega_i^{(j)}:=\sigma_j(\omega_i)$, then we have
	$$\sigma(\sum_i a_i\omega_i)=\left(\sum_i a_i\omega_i^{(1)}, \dots, \sum_i a_i\omega_i^{(n)} \right).$$
	By allowing the coefficients in this formula to be real, i.e., $a_i\in\R$, we extend $\sigma$ to a vector space isomorphism 
	$\sigma:\omega_1\R+\dots+\omega_n\R\simeq \R^n$.
	
	The vectors $v_1,\dots,v_n\in\R^n$ spanning the extreme rays $v_1\R_{\geq 0},\dots,v_n\R_{\geq 0}$ from step (3) correspond to the coordinates in the basis $\nu_1,\dots,\nu_n$ (recall the definition of $\tau$ in step (2)).
	Under the isomorphism $\sigma$, the vector 
	$$v_j=(v_{j1},\dots,v_{jn}) \text{ thus corresponds to the vector }\left(\sum_i v_{ji}\omega_i^{(1)}, \dots, \sum_i v_{ji}\omega_i^{(n)} \right).$$ 
	As $\sigma$ is the Minkowski embedding and we are considering the extreme rays of the cone generated by the totally positive elements, these rays must be just the coordinate half-axes $E_i$.
	We cannot recover the order of the axes (i.e., there is a permutation $\pi$ of $\{1,\dots,n\}$ such that the ray $v_i\R_{\geq 0}$ corresponds to $E_{\pi(i)}$), but permuting the coordinates $x_i$ corresponds to permuting the embeddings $\sigma_i$ -- therefore, without loss of generality, we can just renumber the embeddings $\sigma_i$ so that $\pi=\text{id}$, i.e.,  $v_i\R_{\geq 0}$ corresponds to $E_{i}$.
	
	Then 
	$$\left(\sum_i v_{ji}\omega_i^{(1)}, \dots, \sum_i v_{ji}\omega_i^{(n)} \right)\in E_j,$$
	i.e., 
	\begin{equation*}
		\sum_i v_{ji}\omega_i^{(h)}=0\text{ for all } h\neq j.
	\end{equation*}
	
	Consider now the system of equations above corresponding to $h=1$, i.e.,
	\begin{align*}
		\sum_i v_{2i}X_i&=0\\
		\sum_i v_{3i}X_i&=0\\
		\vdots&\\
		\sum_i v_{ni}X_i&=0		
	\end{align*}
	This is a homogeneous system of $n-1$ equations in $n$ variables $X_1,\dots,X_n$ (with coefficients $v_{ji}\in\R$). We know that the system has a solution $X_i=\omega_i^{(1)}$; moreover, $\omega_i^{(1)}\neq 0$ for all $i$ (as they are images of the $\Q$-basis $\omega_i$ of $K$ under an embedding $\sigma_1$). 
	In step (3) we established that the vectors $v_2,\dots,v_n\in\R^n$ are linearly independent, and so the system has a \textit{unique} solution such that $X_1=1$ (namely, $X_i=\omega_i^{(1)}/\omega_1^{(1)}$).
	
	\medskip
	
	\textbf{(5) Recovering $K$.}\nopagebreak
	
	Thus we have defined the number field
	$K':=\Q\left( X_2,\dots,X_n\right) $. As established in (4), we have $X_i=\omega_i^{(1)}/\omega_1^{(1)}$, but we still do not know the elements $\omega_i^{(1)}$ or the number field $\Q\left( \omega_1^{(1)},\omega_2^{(1)},\dots,\omega_n^{(1)}\right) \simeq K$. 
	
	The elements $\omega_1^{(1)},\omega_2^{(1)},\dots,\omega_n^{(1)}$ are $\Q$-linearly independent, and so also $\omega_1^{(1)}/\omega_1^{(1)}=1,$  $\omega_2^{(1)}/\omega_1^{(1)},\dots,\omega_n^{(1)}/\omega_1^{(1)}$ are $\Q$-linearly independent. Therefore the number field $K'$ has degree at least $n$ over $\Q$. At the same time, $K'$ is a subfield of the degree $n$ field $\Q\left( \omega_1^{(1)},\omega_2^{(1)},\dots,\omega_n^{(1)}\right) \simeq K$, and so we must have 
	$$K'=\Q\left( \omega_2^{(1)}/\omega_1^{(1)},\dots,\omega_n^{(1)}/\omega_1^{(1)}\right)=\Q\left( \omega_1^{(1)},\omega_2^{(1)},\dots,\omega_n^{(1)}\right) \simeq K,$$
	which recovers $K$ as we wanted.
	 Note that this indeed determines $K$ only up to isomorphism, for $\omega_j^{(1)}=\sigma_1(\omega_j)$ (and $\omega_1,\dots,\omega_n$ was a basis of $K$ itself).
\end{proof}

\section{Presentations of semilattices}\label{sec:6}

In Proposition \ref{prop:sum ind} we considered the condition that every element of a semilattice $\mc S$ is a sum of indecomposables. This condition can be reformulated as saying that the set of indecomposables $\mc I$ \textit{generates} $\mc S$ as a semigroup. 

We will now aim to describe all the (additive) relations that hold between indecomposables, in order to fully describe the additive structure of the semilattice.
The right notion for doing this is that of a presentation (coming from universal algebra), which we now introduce. For more background, see  \cite[Chapter I, Section 6]{Gr}.

First, for a nonempty set $X$ (that can be finite or infinite), the \textit{free commutative semigroup} over $X$ is a commutative semigroup $F_X\supset X$
uniquely determined by the following universal property: For every commutative semigroup $S$ and a map $f:X\rightarrow S$, there is a unique semigroup homomorphism $F_X\rightarrow S$ that extends $f$.
Concretely, the free commutative semigroup can be viewed as the set of all the sums of elements of $X$, i.e., 
\[F_X=\left\{\sum_{x\in X} n_x x\mid n_x\in\Z_{\geq 0}, \text{ finitely many coefficients }n_x\text{ are nonzero}\right\}\]
with the addition defined as  $\sum_{x\in X} n_x x+\sum_{x\in X} n_x' x=\sum_{x\in X} (n_x+n_x') x$.
Let us stress that $0\not\in F_X$, i.e., we require that at least one coefficient $n_x$ is nonzero.

It is easy to observe that $F_X$ is cancellative and torsion-free.

\medskip

Let $S$ be a commutative semigroup. A \textit{congruence} $\sim$ on $S$ is an equivalence on $S$ that is preserved under $+$, i.e., for all $x,y,z\in S$ we have $x\sim y\implies x+z\sim y+z$. 
Note that formally, a congruence is a subset of $S\times S$, i.e., ${\sim}  \subset S\times S$.

Given a congruence $\sim$, we can define the \textit{factorsemigroup} $S/{\sim}$ as the set of all blocks $[x]_\sim$ of the equivalence $\sim$ with addition defined as $[x]_\sim+[y]_\sim:=[x+y]_\sim$. This is again a commutative semigroup, and we have a surjective semigroup homomorphism $S\twoheadrightarrow S/{\sim}$.

Let $S$ be a commutative semigroup that is generated by its subset $X\subset S$ (i.e., $S$ is the smallest subsemigroup containing $X$).
Thanks to the universal property, there is a unique semigroup homomorphism $\pi:F_X\rightarrow S$ that fixes $X$. This homomorphism defines a congruence on $F_X$, namely, $x\sim_\pi y\Leftrightarrow \pi(x)=\pi(y)$. We then have a semigroup isomorphism $F_X/{\sim_\pi}\simeq S$.
Thus, every commutative semigroup $S$ is isomorphic to a factor of the free commutative semigroup $F_X$ over 
arbitrary generating set $X$ of $S$.

A \textit{presentation} of a commutative semigroup $S$ is a pair  $\langle X \mid R \rangle$ where $X$ is a set, $R\subset F_X\times F_X$ is a set of relations, and $S\simeq F_X/{\sim_R'}$, where $\sim_R'$ is the smallest congruence on $F_X$ containing $R$. Let us note that here we think of elements of $F_X\times F_X$ as relations, as they tell us which equations hold in the factor $S\simeq F_X/{\sim_R'}$.

As we are interested in semilattices, i.e., commutative semigroups that are also cancellative and torsion-free, it will be convenient for us to work with a notion of presentation that has these properties built in. In the following definition, ``CT'' stands for cancellative, torsion-free.

\begin{definition}\label{de:CT} 
	Let $S$ be a cancellative, torsion-free, commutative semigroup.
	A \textit{CT-congruence} $\sim$ on $S$ is a congruence that satisfies for all $x,y,z\in S$ and $n\in\Z^+$
	\begin{itemize}
		\item[(C)] $x+z\sim y+z\implies x\sim y$, and
		\item[(T)] $nx \sim ny \implies x\sim y$ (where $nx$ denotes the $n$-times iterated addition $x+x+\dots+x$).
	\end{itemize}

Let $X$ be a set and $R\subset F_X\times F_X$ a set of relations.
Then there exists the smallest CT-congruence on $F_X$ containing $R$, namely, the intersection of all CT-congruences that contain $R$. We call it the \textit{CT-congruence generated by $R$}.

Finally,  $\langle X \mid R \rangle_{CT}$ is a 
\textit{CT-presentation} of $S$ if $S\simeq F_X/{\sim_R}$, where $\sim_R$ is the CT-congruence on $F_X$ generated by $R$.
\end{definition}

Let us now go back to our study of semilattices. We are interested in finding a CT-presentation of a semilattice in terms of its indecomposables, for which we need to assume that the indecomposables generate our semilattice (as discussed at the end of Section \ref{sec:2}).

First, observe that if the indecomposables $\mc I$ generate $\mc S$, we have $\spa_{\R}(\mc{I})=\spa_{\R}(\mc{S})=V$, and so there is an $\R$-basis of $V$ consisting of elements of $\mc I$.
We will fix some such basis; the following setting will be in force till the end of the paper.

\begin{settings} \label{set:2}
	Let $V$ be a $\mathbb{R}$-vector space of dimension $n$ and  $\mathcal{S} \subset V$ a semilattice.  
	
	Let $\mc I$ be the set of all indecomposables in $\mc S$, \textbf{assume} that every element of $\mc S$ is a sum of indecomposables, and  fix an $\R$-basis  $\nu_1,\dots,\nu_n\in\mc I$ of $V$.

    For each $\alpha \in \mathcal{I}$ we define a formal variable $x_\alpha$, and we let $I = \{x_\alpha \mid \alpha \in \mathcal{I}\}$ be the set of all formal variables. Let $F_I$ be the free commutative semigroup over $I$.

    For each $\alpha \in \mathcal{I}$, we now consider the unique equation \eqref{eq:Ra} and define the corresponding relation on $F_I$ as $r_\alpha:=(u_\alpha, v_\alpha)\in F_I\times F_I$ where
    \[u_{\alpha} := m x_{\alpha} + \sum_{i=1}^n q_ix_{\nu_i}\text{ and } v_{\alpha} := \sum_{i=1}^n p_ix_{\nu_i}.\]
    
    Let $R = \{r_{\alpha} \mid \alpha \in \mathcal{I}\}\subset F_I\times F_I$. 
 \end{settings}

This gives us the desired CT-presentation of $\mc S$.

\begin{theorem}\label{th:pres1} 
	Assume Setting $\ref{set:2}$. Then
    $\langle I \mid R \rangle_{CT}$ is a CT-presentation of $\mathcal{S}$. 
\end{theorem}

\begin{proof}
     We consider the homomorphism $\tau: F_I \to \mc S$ that is uniquely determined by $\tau (x_\alpha) = 
     \alpha$. The indecomposables $\mathcal{I}$ generate $\mathcal{S}$, and so $\tau $ is surjective. 
     As before, we let $\sim$ be the congruence on $F_I$ defined by $x \sim y$ if and only if $\tau (x) = \tau (y)$, so that we have $F_I / {\sim}\simeq \mc S$.
	As $\mc S$ is cancellative and torsion-free, it is easy to see that $\sim$ is a CT-congruence. As each equation \eqref{eq:Ra} holds in $\mc S$, we have $R\subset {\sim}$.

   Now let ${\sim_R}$ be the CT-congruence on $F_I$ generated by $ R $. Since $\sim$ is a CT-congruence that contains $R$, we have $ {\sim_R} \subset {\sim}$. In order to show that $\langle I \mid R \rangle_{CT}$ is a CT-presentation of $\mathcal{S}$, we need to show that $\mc S\simeq F_I/{\sim_R}$. As $F_I / {\sim}\simeq \mc S$, this amounts to showing $ {\sim_R} = {\sim}$. Thus we need to show $ {\sim_R} \supset {\sim}$, as we already know the other inclusion.
   
   \medskip
   
   Therefore assume that $x\sim y$ for some $x,y\in F_I$; we want to show $x\sim_R y$.
   Let $x=\sum_{\alpha\in\mc I} n_\alpha x_{\alpha}$, $y=\sum_{\alpha\in\mc I} n_\alpha' x_{\alpha}$.
   By the definition of $\sim$, we have $\tau(x)=\tau(y)$ in $\mc S$, i.e., 
   \begin{equation}\label{eq:1}
   	\sum_{\alpha\in\mc I} n_\alpha {\alpha}=\sum_{\alpha\in\mc I} n_\alpha' {\alpha}.
   \end{equation}
   
   We want to now use the equations \eqref{eq:Ra} to rewrite \eqref{eq:1} into an equation that involves only the basis elements $\nu_1,\dots,\nu_n\in\mc I$.
   For that purpose, let first $\mc A\subset\mc I$ be the finite set of indecomposables $\alpha$ for which $n_\alpha$ or $n_\alpha'$ is nonzero (i.e., which appear in \eqref{eq:1}).
      
   Let $M$ be the least common multiple of the coefficients $m$ from all the equations \eqref{eq:Ra} for $\alpha\in\mc A$. We multiply \eqref{eq:1} by $M$, and then add $\sum_{i=1}^n r\nu_i$ to both sides of the resulting equation, for a sufficiently large positive integer $r$; we obtain
     \begin{equation}\label{eq:2}
    	\sum_{\alpha\in\mc I} n_\alpha M{\alpha}+\sum_{i=1}^n r\nu_i=\sum_{\alpha\in\mc I} n_\alpha' M{\alpha}+\sum_{i=1}^n r\nu_i.
    \end{equation}
    The integer $r$ is chosen large enough so that we can repeatedly apply \eqref{eq:Ra} for all $\alpha\in \mc A$ to both sides of \eqref{eq:2} in order to replace all the occurrences of $\alpha$ in this equation (that is also why we multiplied by $M$ that is divisible by all the coefficients $m$), so that we arrive at an equation of the form
    \begin{equation}\label{eq:3}
   \sum_{i=1}^n s_i\nu_i=\sum_{i=1}^n t_i\nu_i
   \end{equation}
   for some $s_i,t_i\in\Z_{\geq 0}$.
   This equation holds in $\mc S\subset V$, and $\nu_1,\dots,\nu_n$ form a basis of $V$, and so $s_i=t_i$ for all $i$.
   
   \medskip
   
   Now we can proceed backwards in order to deduce $x\sim_R y$. We start with the tautological relation 
   $\sum_{i=1}^n s_ix_{\nu_i}\sim_R \sum_{i=1}^n s_ix_{\nu_i}$, to which we apply the relations $r_\alpha$ in exactly the opposite way as we did going from \eqref{eq:2} to \eqref{eq:3}. This leads to the analogue of \eqref{eq:2}, namely,
   \begin{equation}\label{eq:4}
    	\sum_{\alpha\in\mc I} n_\alpha Mx_{\alpha}+\sum_{i=1}^n rx_{\nu_i}\sim_R \sum_{\alpha\in\mc I} n_\alpha' Mx_{\alpha}+\sum_{i=1}^n rx_{\nu_i}.
    \end{equation}
    
   As $\sim_R$ is a CT-congruence, we can use the property (C) of Definition \ref{de:CT} to ``subtract'' $\sum_{i=1}^n rx_{\nu_i}$ from both sides of \eqref{eq:4}, and then use the property (T) to ``divide'' by $M$, so that we finally arrive at
      \begin{equation}\label{eq:5}
   	x=\sum_{\alpha\in\mc I} n_\alpha x_{\alpha}\sim_R \sum_{\alpha\in\mc I} n_\alpha' x_{\alpha}=y,
   \end{equation}
   as we wanted.
\end{proof}

\section{Finitely generated set of relations}\label{sec:7}

While the presentation given in Theorem \ref{th:pres1} is concrete, it has the disadvantage that it involves one relation $r_\alpha$ for each indecomposable $\alpha$, with no way of a priori knowing the values of the coefficients $m,p_i,q_i$ that come from \eqref{eq:Ra}.
In a general semilattice, it is unclear whether we can do better. But the semilattices defined in totally real number fields have only finitely many indecomposables up to multiplication by totally positive units, which suggests that there may be a finite set of relations that, in a suitable sense, gives the presentation up to multiplication by units.

In order to formalize this, we formulate things somewhat more generally, in terms of automorphisms acting on a semilattice $\mc S$ (that, in the number field setting, are given precisely by the multiplication by units).

\begin{settings}\label{set:3}
    Assume Setting \ref{set:2}. Let $\Phi_0$ be a set of semigroup automorphisms of $\mathcal{S}$ that contains the identity and is closed under inverses (i.e., $\text{id}\in \Phi_0$ and if $\varphi\in\Phi_0$, then $\varphi^{-1}\in\Phi_0$). Let $\Phi$ be the group generated by $\Phi_0$ 
    and let 
    \[\left\{g(\nu_i)\mid g\in\Phi_0, i=1,\dots,n\right\}\subset \mc I_0 \subset \mc I\] be a set of indecomposables that is a fundamental set in the sense that for each $\alpha\in\mc I$, there are (not necessarily unique) $\varphi\in \Phi$ and $\alpha_0\in\mc I_0$ such that $\alpha=\varphi(\alpha_0)$. 
    
    For $x=\sum_{\alpha\in\mc I} n_\alpha x_{\alpha}\in F_I$ and $\varphi\in\Phi$, define
    $\varphi(x):=\sum_{\alpha\in\mc I} n_\alpha x_{\varphi(\alpha)}\in F_I$. 
    
    For a relation $r=(u,v)\in F_I\times F_I$, we define $\varphi(r):=(\varphi(u),\varphi(v))$.

We can now define a new set of relations 
\[
R':=
\{\varphi(r_{\alpha})\mid \varphi\in\Phi,\alpha\in\mc I_0\}.
\]
\end{settings}

\medskip

\begin{example} \label{ex:nf}
	Let us follow Section \ref{sec:4} and all the notation introduced there. 
The semilattice $\mc S=\sigma(\co_K^+)$ of totally positive integers in a totally real number field $K$ precisely fits in Setting \ref{set:3}: 
We can take $\Phi=\co_K^{\times,+}$, with its free generators $\ve_1,\dots,\ve_{n-1}$, and let $\Phi_0=\{\text{id},\ve_1^{\pm1}, \dots,\ve_{n-1}^{\pm1}\}$. We view $\Phi$ as automorphisms of $\mc S$ through the action of multiplication by units (concretely, the action of $\ve\in\Phi$ on $\alpha\in\mc S\subset(\R^+)^n$ sends it to the coordinate-wise product $\sigma(\ve)\alpha\in\mc S\subset(\R^+)^n$).

Since there are finitely many indecomposables up to multiplication by $\co_K^{\times,+}$, there indeed is a \textit{finite} fundamental set $\mc I_0' \subset \mc I$, which we can enlarge so that it contains the elements $g(\nu_i)$ for $g\in\Phi_0, i=1,\dots,n$, yielding the desired fundamental set $\mc I_0$.

In this case, $\Phi_0$ and $\mc I_0$ are both finite.
\end{example}

Note that in the setting, we have not posed any assumptions concerning the finiteness of $\Phi_0$  or $\mc I_0$ (as we do not need them). However, when both $\Phi_0$  and $\mc I_0$ are finite (as is the case in the number field Example \ref{ex:nf}), then 
 the set $R'$ is indeed finite up to the action of $\Phi$, as it consists of all the translates by $\Phi$ of the finitely many relations $r_\alpha$, $\alpha\in\mc I_0$.
 
 Also note that when the set of indecomposables $\mc I$ is infinite, then we cannot do much better, as each indecomposable must appear in some relation, and so the generating set of relations for every CT-presentation must be infinite. However, our fundamental set $\mc I_0$ of course could be made smaller (as it currently contains the equivalent elements $\nu_1,\varphi(\nu_1),\varphi^{-1}(\nu_1)$ for each $\varphi\in\Phi_0$, e.g.). An open question remains whether (and how) it is possible to do this; this may of course involve considering different relations than our $r_\alpha$.

\begin{theorem}  \label{th:pres2}
	Assume Settings $\ref{set:2}$ and $\ref{set:3}$. Then
    $\langle I \mid R' \rangle_{CT}$ is a CT-presentation of $\mathcal{S}$.
\end{theorem}

\begin{proof}
Let $\sim, \sim'$ be the CT-congruences generated by $R, R'$, respectively.	By Theorem \ref{th:pres1}, we have $\mc S\simeq F_I/{\sim}$, and so we need to show that ${\sim}={\sim}'$.
	For that, it suffices to show that $R\subset{\sim}'$ and $R'\subset{\sim}$.

	The inclusion $R'\subset{\sim}$ holds obviously, as all the equations corresponding to the relations in $R'$ under the isomorphism $F_I/{\sim}\simeq\mc S$ hold in $\mc S$.
	
	For the other inclusion, let us show that $r_\alpha\in{\sim'}$ for all $\alpha\in\mc I$.
	Since $\mc I_0$ is a fundamental set for the action by $\Phi$, there is $\alpha_0\in\mc I_0$ and $\varphi\in\Phi$ such that $\alpha=\varphi(\alpha_0)$.
	Since $\Phi$ is generated by $\Phi_0$ (which is closed under inverses), there are $g_1,g_2,\dots, g_k \in \Phi_0$ such that $\alpha = g_1   \cdots   g_k (\alpha_0)$. 

Let us show that $r_\alpha\in{\sim'}$ by induction on the length of the expression $k$.

  If $k=0$, then $\alpha=\alpha_0\in\mc I_0$. Thus we directly have $r_\alpha=\text{id}(r_\alpha)\in R'$.
  
  For the induction step, let us fix $k\geq 1$ and assume that for all $g_1,\dots,g_{k-1}\in\Phi_0$ and all $\alpha_0\in\mc I_0$ 
  we have
  $r_{g_{1}   \dots   g_{k - 1} (\alpha_0)}\in{\sim'}$. 
 We want to show that $r_\alpha\in{\sim'}$ for our $\alpha = g_1  \cdots  g_k (\alpha_0)$.

 Let us start with the relations
 \begin{align*}
 	r_{\alpha_0} &=: \left( m_0 x_{\alpha_0} + \sum_{i=1}^n q_{0i}x_{\nu_{i}}, \sum_{i=1}^n p_{0i}x_{\nu_i}\right) \in R' \text{ and} \\
 	g_{1}   \cdots   g_k (r_{\alpha_0}) &= \left( m_0 x_{\alpha} + \sum_{i=1}^n q_{0i}x_{g_{1}   \cdots   g_k (\nu_{i})}, \sum_{i=1}^n p_{0i}x_{g_{1}   \cdots   g_k (\nu_i)}\right) \in R'.
 \end{align*}
 
 We also have for each $j$
  \begin{align*}
 	r_{g_k(\nu_j)} &=: \left( m_j x_{g_k(\nu_j)} + \sum_{i=1}^n q_{ji}x_{\nu_{i}}, \sum_{i=1}^n p_{ji}x_{\nu_i}\right) \in R' \text{ and} \\
 	g_{1}   \cdots   g_{k-1} (r_{g_k(\nu_j)}) &= \left( m_j x_{g_{1}   \cdots   g_k(\nu_j)} + \sum_{i=1}^n q_{ji}x_{g_{1}   \cdots   g_{k-1} (\nu_{i})}, \sum_{i=1}^n p_{ji}x_{g_{1}   \cdots   g_{k-1} (\nu_i)}\right) \in R'.
 \end{align*}

 By the induction hypothesis, the relations $r_{g_{1}   \cdots   g_{k-1} (\nu_i)}\in{\sim'}$ for all $i$. 
  As in the proof of Theorem \ref{th:pres1}, we can multiply both sides of $g_{1}   \cdots   g_{k-1} (r_{g_k(\nu_j)})$ by $M$ and then add $\sum_{i=1}^n rx_{\nu_i}$ to both sides of the resulting relation, for sufficiently large positive integers $r,M$, so that we can apply $r_{g_{1}   \cdots   g_{k-1} (\nu_i)}$ to both sides in order to rewrite the relation as
 \begin{align}\label{eq:6.1}
 	& \left( m_j' x_{g_{1}   \cdots   g_k(\nu_j)} + \sum_{i=1}^n q_{ji}'x_{\nu_{i}}, \sum_{i=1}^n p_{ji}'x_{\nu_i} \right) \in {\sim'},	
 \end{align}
 for some $m_j'\in\Z^+$ and $p_{ji}', q_{ji}'\in\Z_{\geq0}$.

 Now we want to ``plug-in'' these relations into the relation $g_{1}   \cdots   g_k (r_{\alpha_0})\in R'$ that we obtained above. For that purpose, we again suitably multiply and add to both sides of $g_{1}   \cdots   g_k (r_{\alpha_0})$, so that we can apply \eqref{eq:6.1} on both sides in order to rewrite the relation as 
 \[
 \left( m' x_{\alpha} + \sum_{i=1}^n q_{i}'x_{\nu_{i}}, \sum_{i=1}^n p_{i}'x_{\nu_i}\right) \in {\sim'}.
 \]

 This relation has the same shape as $r_{\alpha}= \left( m x_{\alpha} + \sum_{i=1}^n q_{i}x_{\nu_{i}}, \sum_{i=1}^n p_ix_{\nu_i}\right) \in F_I\times F_I$ and it holds in $\mc S$ (for all relations in $R'$ hold there, and thus also all relations in $\sim'$ hold). Thus we can use the uniqueness result of Lemma \ref{lem:rel2} which says that
 there exist $\ell \in\Z^+,r_i \in\Z_{\geq 0} $ such that $m' = \ell m$ and $p_i' = \ell p_i + r_i$ and $q_i' = \ell q_i + r_i$.
 
 In other words, we have 
 \begin{equation}\label{eq:6.2}
 	\left( \ell m x_{\alpha} + \sum_{i=1}^n (\ell q_i + r_i)x_{\nu_{i}}, \sum_{i=1}^n (\ell p_i + r_i)x_{\nu_i}\right) \in {\sim'}. 	
 \end{equation}
 As $\sim'$ is a CT-congruence, we can use first property (C) from Definition \ref{de:CT} to cancel 
 the term $\sum_{i=1}^n  r_ix_{\nu_i}$ on both sides of relation \eqref{eq:6.2}, and then the property (T) to divide by $\ell $. Doing that, we arrive at 
 \[r_{\alpha}= \left( m x_{\alpha} + \sum_{i=1}^n q_{i}x_{\nu_{i}}, \sum_{i=1}^n p_ix_{\nu_i}\right) \in {\sim'},\]
 as we wanted.
\end{proof}


\begin{thebibliography}{abcde}
	
	\bibitem[AF+]{AF+} I. Araujo, B. Frederickson, R. A. Krueger, B. Lidický, T. B. McAllister, F. Pfender, S. Spiro, E. N. Stucky, Triangle percolation on the grid, Discrete Comput. Geom. \textbf{73} (2025), 569--593


     


    \bibitem[BH]{BH} M. Bhargava, J. Hanke, \emph{Universal quadratic forms and the 290-theorem}, preprint



    \bibitem[BK]{BK} V. Blomer, V. Kala, \emph{Number fields without universal $n$-ary quadratic forms}, Math. Proc. Cambridge Philos. Soc. \textbf{159} (2015), 239--252



    \bibitem[Br]{Br} H. Brunotte, \textit{Zur Zerlegung totalpositiver Zahlen in Ordnungen totalreeller algebraischer Zahlkörper}, Arch. 	Math. (Basel) \textbf{41} (1983), 502--503

    \bibitem[CdS+]{CdSL+} G. Cornelissen, B. de Smit, X. Li, M. Marcolli, H. Smit, \emph{Characterization of global fields by Dirichlet $L$-series}, Res. Number Theory \textbf{5} (2019), Paper No.~7, 15 pp.
	
	\bibitem[CO]{CO} W. K. Chan, B.-K. Oh,   \emph{Can we recover an integral quadratic form by representing all its subforms?}, Adv. Math. \textbf{433} (2023),  Paper No. 109317, 20 pp.

    \bibitem[Cu]{Cu} F. Curtis, \emph{On formulas for the Frobenius number of a numerical semigroup}, Math. Scand. \textbf{67} (1990), 190--192
	

    \bibitem[De]{De} C. Delorme, \emph{Sous-mono\"\i des d'intersection compl\`ete de $\mathbf{N}$}, Ann. Sci. \'Ecole Norm. Sup. (4) \textbf{9} (1976),  145--154

    \bibitem[DS]{DS} A. Dress, R. Scharlau, \emph{Indecomposable totally positive numbers in real quadratic orders}, J. Number Theory \textbf{14} (1982),  292--306
	

    \bibitem[FK]{FK} L. Fukshansky, F. Kostopoulou, \emph{On indecomposable elements in lattices},  \href{https://arxiv.org/abs/2606.00868}{arxiv:2606.00868} 
	


    

    \bibitem[Gr]{Gr} P. A. Grillet, \emph{Commutative semigroups}, Advances in Mathematics (Dordrecht) \textbf{2}, Kluwer Academic Publishers, Dordrecht, 2001
	


    \bibitem[HK]{HK} T. Hejda, V. Kala, \emph{Additive structure of totally positive quadratic integers}, Manuscripta Math. \textbf{163} (2020), 263--278



    \bibitem[Ka]{Ka} V. Kala, \emph{Universal quadratic forms and indecomposables in number fields: A survey}, Commun. Math. \textbf{31} (2023), Special issue: Euclidean lattices: theory and applications, 81--114
	
	

    \bibitem[KKR]{KKR} V. Kala, J. Kr\'asensk\'y, G. Romeo, \emph{Universality criterion sets for quadratic forms over number fields}, Adv. Math. \textbf{500} (2026), Paper No.~111080, 27 pp.
	
	

    \bibitem[KM]{KM} V. Kala, S. H. Man, \emph{Sails for universal quadratic forms}, Selecta Math. (N.S.) \textbf{31} (2025),  Paper No.~26, 31 pp.

    \bibitem[Ko]{Ko} K. Komatsu, \emph{On the adele rings of algebraic number fields}, K\=odai Math. Sem. Rep. \textbf{28} (1976), 78--84


    \bibitem[KT]{KT} V. Kala, M. Tinkov\'a, \emph{Universal quadratic forms, small norms, and traces in families of number fields}, Int. Math. Res. Not. IMRN (2023), 7541--7577

    \bibitem[KY]{KY3} V. Kala, P. Yatsyna, \emph{On Kitaoka's conjecture and lifting problem for universal quadratic forms},  Bull. Lond. Math. Soc. \textbf{55} (2023), 854--864

	
	   \bibitem[KYZ]{KYZ} V. Kala, P. Yatsyna, B. \. Zmija, \emph{Real quadratic fields with a universal form of given rank have density zero}, Amer. J. Math. (to appear), \href{https://arxiv.org/abs/2302.12080}{arxiv:2302.12080}
	
	
\bibitem[Man]{Man} S. H. Man, \emph{Minimal rank of universal lattices and number of indecomposable elements in real multiquadratic fields}, Adv. Math. \textbf{447} (2024), Paper No. 109694, 38 pp.
	

    \bibitem[MS]{MS} G. Mantilla-Soler, \emph{Density questions on arithmetic equivalence}, Proc. Amer. Math. Soc. \textbf{151} (2023),  2783--2794
	
	 \bibitem[Na]{Na} W. Narkiewicz, \emph{Elementary and analytic theory of algebraic numbers}, 3rd Edition, Springer-Verlag, Berlin, 2004




     \bibitem[Neu]{Neu} J. Neukirch, \emph{Kennzeichnung der $p$-adischen und der endlichen algebraischen Zahlk\"orper}, Invent. Math. \textbf{6} (1969), 296--314


     


    \bibitem[Pe]{Pe} O. Perron, \emph{Die Lehre von den Kettenbr\"uchen}, B. G. Teubner, Leipzig, 1913

    \bibitem[Pr]{Pr} D. Prasad, \emph{A refined notion of arithmetically equivalent number fields, and curves with isomorphic Jacobians}, Adv. Math. \textbf{312} (2017), 198--208

    \bibitem[PSZ]{PSZ} E. Pěchoučková, D. Stern, M. Zindulka, \emph{Extending integer partition identities to semigroups}, preprint


    \bibitem[RB]{RB} J. C. Rosales, M. B. Branco, \emph{Irreducible numerical semigroups}, Pacific J. Math. \textbf{209} (2003),  131--143

    


    \bibitem[RG]{RG} J. C. Rosales, P. A. Garc\'ia-S\'anchez, \emph{Numerical semigroups}, Developments in Mathematics, 20, Springer, New York, 2009

    
	


    \bibitem[Ro]{Ro} R. T. Rockafellar, \emph{Convex analysis}, Princeton Mathematical Series, No. 28, Princeton University Press, Princeton, N.J., 1970
	

	
	   \bibitem[Si]{Si3} C. L. Siegel, \emph{Sums of $m$-th powers of algebraic integers},  Ann. of Math. \textbf{46} (1945),  313--339 
	

    \bibitem[Su]{Su} A. V. Sutherland, \emph{Stronger arithmetic equivalence}, Discrete Anal. \textbf{2021}, Paper No.~23, 23 pp.

    \bibitem[Ti]{Ti} M. Tinkov\'a, \emph{On the number of indecomposable integers in totally real cubic fields}, J. Number Theory \textbf{232} (2022), 36--63

    \bibitem[TV]{TV} M. Tinkov\'a, P. Voutier, \emph{Indecomposable integers in real quadratic fields}, J. Number Theory \textbf{212} (2020), 458--482



    \bibitem[Uch]{Uch} K. Uchida, \emph{Isomorphisms of Galois groups}, J. Math. Soc. Japan \textbf{28} (1976), 617--620


    \bibitem[Wi]{Wi} H. S. Wilf, \emph{A circle-of-lights algorithm for the money-changing problem}, Amer. Math. Monthly \textbf{85} (1978),  562--565
	

    
	
	 \bibitem[Ya]{Ya} P. Yatsyna, \emph{A lower bound for the rank of a universal quadratic form with integer coefficients in a totally real field}, Comment. Math. Helvet. \textbf{94} (2019), 221--239

\end{thebibliography}
\end{document}